\documentclass[12pt]{amsart}
\usepackage{amsmath,amssymb, euscript, graphicx, hyperref}

\usepackage{latexsym,array,delarray,amsthm,amssymb,epsfig,setspace,tikz,amsmath,enumerate,mathrsfs,graphicx, mathtools,cite}
\usepackage{geometry}
\usepackage{tabularx}

\newtheorem{thm}{Theorem}[section]
\newtheorem{lem}[thm]{Lemma}
\newtheorem{prop}[thm]{Proposition}
\theoremstyle{definition}

\newtheorem{rem}[thm]{Remark}
\newtheorem{cor}[thm]{Corollary}

\newcommand{\blackboard}[1]{\ensuremath{\mathbb{#1}}}

\newcommand{\Z}{\blackboard{Z}}

\begin{document}

\title{Girth Alternative for HNN Extensions}
\author{Azer Akhmedov, Pratyush Mishra}

\address{Department of Mathematics
\\North Dakota State University\\}

\email{azer.akhmedov@ndsu.edu, pratyush.mishra@ndsu.edu}

%\keywords{keyword 1, keyword 2}

\maketitle

\begin{abstract}
    We prove the Girth Alternative for a sub-class of HNN extensions as well as for a sub-class of amalgamated free products of finitely generated groups, and indicate counterexamples to show that beyond our class, the alternative fails in general.We also prove the Girth Alternative for HNN extensions of non-elementary word hyperbolic groups.
\end{abstract}

\maketitle
\section{introduction}

Let $G=\langle S | R \rangle$ be a finitely generated group and let $A$ and $B$ be two subgroups of $G$ with isomorphism $\phi: A\to B$. The \textit{HNN extension group of G relative to subgroups A and B} with stable letter $t$, denoted by $(G,A,B,t)$ is the extended group containing $G$ defined as $$G^{\ast}_{\phi}=(G,A,B,t)=\langle G, t| t^{-1}at=\phi(a) \hspace{0.1cm} \text{for} \hspace{0.1cm} a\in A\rangle$$ where $A$ and $B$ are not only isomorphic but also conjugate via the map $\phi$. The HNN extensions originally arose in topology as the fundamental groups of a topological space when two subspaces are glued along a homeomorphism. \footnote{In this case one obtains a somewhat more general notion where $A$ and $B$ are not necessarily subgroups of $G$ but have homomorphic images in $G$.} In recent decades, HNN extensions have been used as a popular tool to construct examples and counterexamples of groups for questions in combinatorial group theory. 
 \medskip
 
In \cite{S}, Schleimer defined the girth of $G$ with respect to a finite generating set $S$, denoted by $girth(G,S)$ as the length of the shortest non-trivial relation in $G$ with respect to generating set $S$. And the girth of the group $G$ is defined as $$girth(G)=\sup_{S\subset G} \{girth(G,S)| \langle S \rangle =G, |S| < \infty \}$$

For a given group $G$, it is natural to ask whether the $girth(G)$ is finite or infinite. In \cite{S}, \cite{azer1}, it is shown that finitely generated groups satisfying a law, which are not isomorphic to $\mathbb{Z}$, have finite girth. Moreover, in \cite{azer2}, the author shows that for many classes of groups (word hyperbolic, one-relator, linear groups not isomorphic to $\mathbb{Z}$), the property of having infinite girth coincides with the property of containing a non-abelian free subgroup and introduced the notion of Girth Alternative similar in spirit to the well-known Tits Alternative. \textit {For a given class $\mathcal{C}$ of finitely generated groups, $\mathcal{C}$ is said to satisfy the \textit{Girth Alternative} if any group from the class $\mathcal{C}$ has either infinite girth or is virtually solvable.}

\medskip
 
In \cite{azer2}, \cite{azer3}, Akhmedov has proved the Girth Alternative for the class of hyperbolic, linear, one-relator and $PL_+(I)$ groups. In \cite{Y}, Yamagata proves the Girth Alternative for convergence groups and irreducible subgroups of the mapping class groups. Independently in \cite{N}, Nakamura proves the Alternative for all subgroups of mapping class groups and also for the subgroups of Out$(\mathbb{F}_n)$ containing the irreducible elements having irreducible powers.

\medskip 

The table below shows a dichotomy between Girth Alternative and Tits Alternative for some classes of finitely generated groups that we are interested in.    
\medskip

\begin{tabularx}{1.0\textwidth} { 
  | >{\raggedright\arraybackslash}X 
  | >{\centering\arraybackslash}X 
  | >{\raggedleft\arraybackslash}X | }
\hline
\textbf{Groups} & \textbf{Tits Alternative} & \textbf{Girth Alternative}\\
\hline
PL$_+$(I) & fails (Thompson's group F) & holds\\
\hline
Linear & holds & holds\\
\hline
1-relator & holds & holds \\
\hline 
Hyperbolic & holds & holds \\
\hline
Homeo$_+(I)$ & fails (Thompson's group F) & fails (we prove in \cite{AM})\\
\hline
Diff$^{\omega}_+(I)$ & \textbf{unknown} & holds (we prove in \cite{AM}) \\
\hline
Residually finite & fails & fails \\ 
\hline 
Group of formal power series (over field $\mathbf{k}$) & fails for char$\mathbf{k}>0$ \& \textbf{unknown} for char$\mathbf{k}=0$ & \textbf{fails} for char$\mathbf{k}>0$ \& holds for char$\mathbf{k}$=0 \\
\hline
HNN Extensions & holds for proper extensions & holds for proper extensions (we prove in this paper)\\
\hline

\end{tabularx}
\medskip

 In \cite{CM}, the authors construct an example of a finitely generated residually $p$-group which is not virtually solvable, but satisfies a law. This example shows that Tits Alternative and Girth Alternative fail in the class of residually finite groups. Moreover, since every countably based pro-$p$ group embeds  into the group of formal analytic power series over $\mathbb{F}_p$ \cite{C}, both alternatives fail in the latter group as well. 
 
 \medskip 
 
In this paper, our main result is the Girth Alternative for a sub-class of HNN extensions, showing for these sub-classes, again the property of having infinite girth coincide with the property of containing a non-abelian free subgroup.
\subsection{Conventions.}
 For a given group $G$, we say the HNN extension $(G,A,B,t)$ is \begin{enumerate}
     \item {\em Proper}, when both the underlying subgroups $A$ and $B$ are proper in $G$.
     \item {\em Semi-proper}, when one of the subgroups is proper and the other is the full group $G$.
     \item {\em Full}, when both $A$ and $B$ are full group $G$. \footnote{our terminology may differ from the terminology of many other authors; for example, a semi-proper HNN extension is often called (e.g. in \cite{Bu1}, \cite{Bu2}) an ascending (or strictly ascending) HNN extension in many sources; in \cite{Bu2}, ``proper'' is used for HNN extensions which are either proper or semiproper in our terminology.} 
\end{enumerate}

\begin{thm} \label{thm:proper}
For $G$ be a finitely generated group with $A$ and $B$ two proper subgroups then $girth(\Gamma)=\infty$, where $\Gamma=(G,A,B,t)$ is a proper HNN extension of $G$ relative to $A, B$ and $\phi$.
\end{thm}

Since proper HNN extensions are never solvable (they contain a subgroup isomorphic to $\mathbb{F}_2$), Theorem \ref{thm:proper} implies Girth Alternative for proper HNN extensions.  However, the following result provides a class of counterexamples to show that beyond our sub-class as in Theorem \ref{thm:proper}, the alternative fails in general.

 \begin{prop} \label{thm:semiproper}
For $G$ a finitely generated group satisfying a law with $A=G$ and $B$ a proper subgroup of $G$, then $girth(\Gamma)<\infty$, where $\Gamma=(G,A,B,t)$ is a semi proper HNN extension relative to $G, B$ and $\phi$.
\end{prop}
 
The proof of this proposition is easier; we simply observe that  under the given hypothesis, $\Gamma=(G,A,B,t)$ will satisfy a law and will not be cyclic. In Section 4, we discuss a particulaly interesting case when $G$ is a nilpotent group. 

\medskip

Treating the case of amalgamated free product of groups uses significantly different ideas; we provide a separate proposition devoted to this case.  

\begin{prop} \label{thm:amalgamation} Any proper amalgamated free product $A*_CB$, where $A, B$ are finitely generated groups and  $\max \{A:C, B:C\}\geq 3$ , has infinite girth.   
\end{prop} 

We call an amalgamated free product proper if $C$ is a proper subgroup in both $A$ and $B$. Notice that without the condition $\max \{A:C, B:C\}\geq 3$ the claim does not hold since $girth(D_{\infty }) = 2 < \infty $. \footnote{We use the notation $D_{\infty }$ for the infinite dihedral group $\langle a, b \ | \ a^2 = b^2 = 1\rangle $. It is an infinite virtually cyclic group. Recall that a finite dihedral group $D_n, n\geq 2$ is given by the presentation $\langle a, b \ | \ a^2 = b^2 = (ab)^n = 1\rangle $. The term {\em dihedral group} will refer to $D_q$ where $q\in (\mathbb{N}\backslash \{0,1\})\cup \ \{\infty ]$.}  

\medskip 

Interestingly, we also obtain the following result which shows that the Girth Alternative holds in general for the class of any HNN extension (proper, semi proper or full) of the non-abelian free group $F_n$ for $n\geq 2$.

\begin{prop} \label{thm:free}
Any HNN extension of a non-elementary word hyperbolic group has infinite girth.   
\end{prop}

As a special case we obtain the following 
\begin{cor} \label{thm:free1}
Any HNN extension of the non-abelian free group $\mathbb{F}_n$ for any $n\geq 2$ has infinite girth.   
\end{cor}

 Interestingly, the claim of this corollary does not seem to lend itself to more elementary methods, or to known results in the literature about HNN extension of free groups. We discuss and emphasize some relevant problematic  issues in Section 6. Notice that in the case of $n=1$, Corollary  \ref{thm:free1} easily fails even for semi-proper HNN extensions since the Baumslag-Solitar group $BS(1,m) = \langle a, b \ | aba^{-1} = b^m\rangle $ is solvable hence has finite girth. On the other hand, proper HNN extensions of $\mathbb{Z}$ are all non-solvable one-relator groups, hence, by Theorem 3.1 in \cite{azer2}, have infinite girth. Let us also point out that HNN extensions of a free group $\mathbb{F}_k$ are not necessarily linear; for $k=1$, recall that the groups $BS(n,m) = \langle a, b \ | ab^na^{-1} = b^m\rangle $ are non-Hopfian hence not linear for all $n, m > 1, (m, n) = 1$, and for $k\geq 2$, examples are provided in \cite{DS}

\section{Preliminary results}
First, we prove the following proposition which seems interesting to us also from a purely combinatorial point of view. 

\begin{prop} \label{thm:AB}
Let $G$ be a finitely generated group such that no quotient of $G$ is isomorphic to a dihedral group $D_{n}, n\in (\mathbb{N}\backslash \{0,1\})\cup \{\infty\}$, and let $A, B$ be proper isomorphic subgroups of $G$. Then $G$ admits a finite generating set $S$ such that $S\cap(A\cup B)=\emptyset$.
\end{prop}
\begin{proof}
Let $d=d(G)$ be the minimal cardinality of a generating set of $G$ and $$\mathcal{S}=\{S\subset G| |S|=d,\langle S\rangle =G\}.$$ For cyclic groups the claim is obvious (recalling that a subgroup of a cyclic group is cyclic so we will assume that $G$ is non-cyclic. Then we have $d\geq 2$. We introduce the following quantities: 

$$\alpha(S)=|S\cap (A\backslash B)|, \beta(S)=|S\cap (B\backslash A)|$$ $$\gamma(S)=|S\cap (A\cap B)|, \delta(S)=|S\backslash(A\cup B)|$$

We now claim that there exists a finite generating set $S\in \mathcal{S}$ such that $\delta(S)= d$ (i.e. $\alpha(S)=0, \beta(S)=0,\gamma(S)=0$). Indeed, let $S\in \mathcal{S}$ such that $\delta(S)$ is maximal. Assume $\delta(S)<d$.
\medskip

\textbf{Claim 1:} $\alpha(S)=0$ or $\beta(S)=0.$

Proof: Indeed, assume that $\alpha(S)\geq 1$ and $\beta(S)\geq 1,$ with $s_1\in S\cap(A\backslash B)$ and $s_2\in S\cap (B\backslash A)$. Then replace $S$ with $S'=(S\backslash \{s_1\})\cup \{s_1s_2\}$. Since $s_1s_2\not \in A\cup B$, we obtain that $\delta(S')=\delta(S)+1$, contradicting maximality of $\delta(S)$.

Thus, without loss of generality, we may and will assume that $\beta(S)=0$. Notice that $\delta(S)>0$ because $A$ is a proper subgroup of $G$.
\medskip

\textbf{Claim 2:} $\alpha(S)+\gamma(S)\leq 1$

Proof. For $S$ with $\beta (S) = 0$, suppose $\alpha(S)+\gamma(S)>1$, and let $s_1,s_2\in S\cap A, s_1\neq s_2$. Let also $s_3\in S\backslash (A\cup B)$, then $s_1s_3\not \in A$ but by maximality of $\delta (S)$ \ $$s_1s_3\not \in A\implies s_1s_3\in B.$$ Replace $S$ by $S'=\{s_1s_3,s_2,s_3,\dots, s_n\}$, note that as $\beta(S')=1\neq 0$ then by Claim 1 we have $\alpha(S')=0$, which forces $s_2\in S\cap (A\cap B)$. Similarly, let $S''=\{s_2s_3,s_1,s_3,\dots, s_n\}$ and symmetrically, we obtain that $s_1\in S\cap (A\cap B)$. Then $s_1s_3\notin B$. Contradiction.

\medskip

\textbf{Claim 3:} $\alpha(S)+\gamma(S)=0$, unless $G$ has a quotient isomorphic to a quotient of $D_{\infty }$.

Proof: Assuming the contrary, let $\alpha(S)+\gamma(S)=1$ by Claim 2. Hence $|S\cap A|=1$, so let $S\cap A=\{a\}$ and $S\backslash A=\{s_1,\dots,s_{d-1}\}$. Notice that for all $1\leq i\leq d-1$, $$as_i\not \in A, as_i\in B$$ Choose $s_j\in S$ such that $s_j\not \in B$ (such an $s_j$ exists otherwise that would lead to $\delta(S)=0$; but specifically, we have already assumed that $\delta (S)$ is maximal, $\beta (S) = 0$ and  $\alpha(S)+\gamma(S)=0$, so $s_j\notin B$ for all $1\leq j\leq d-1$). Then, for all $1\leq i\leq d-1$, $$as_is_j\in A, as_is_j\not \in B$$ So, using maximality of $\delta(S),$ inductively on the length of $m\geq 1$ of a reduced word $w=w(s_1,s_2,\dots,s_{d-1})$ in the alphabet set $\{s_1^{\pm 1},\dots, s_{d-1}^{\pm 1}\}$, it follows that $$aw\not \in A, aw\in B \hspace{0.1cm}\text{if $m$ is odd},$$ $$aw\in A, aw\not \in B\hspace{0.1cm} \text{if $m$ is even}.$$ Now, let $\mathcal{W}$ be the set of reduced words in the alphabet $\{s_1^{\pm 1},\dots, s_{d-1}^{\pm 1}\},$ $$H=\{g\in G|g=w(s_1,\dots,s_{d-1})\hspace{0.1cm} \text{such that} \hspace{0.1cm} w\in \mathcal{W}\},$$ $$H_1=\{g\in H|g=w(s_1,\dots,s_{d-1})\hspace{0.1cm}\text{can be written as a word of even length in $H$ in} \hspace{0.1cm} s_1^{\pm 1},\dots,s_{d-1}^{\pm 1}\}$$
\medskip

Then, $[H:H_1]\leq 2$, hence $H_1\trianglelefteq H$. Also, $aH_1\subseteq A$, but note that for all $i\in \mathbb{Z}$,
\begin{equation}
    aH_1\subseteq A\implies H_1\subseteq A \implies a^iH_1a^{-i}\subseteq A
\end{equation} Then as, $H_1\trianglelefteq H$, we also have for all $1\leq i \leq d-1$,
\begin{equation}
    s_iH_1s_i^{-1}\subseteq H_1\subseteq A 
\end{equation} 
Any word in $H_1$ can be written as product of odd words $u^{-1},v \in H$ and we know that $au,av\in B$ so we have, 
\begin{equation}
    u^{-1}v=u^{-1}a^{-1}av=(au)^{-1}av \in B\implies H_1\subseteq B
\end{equation}
Similarly, $aH_1a^{-1}\subseteq B$ and for any $u$ odd word in $H$\begin{equation}
    a^2=auu^{-1}a=au(a^{-1}u)^{-1}\in B
\end{equation}

 Finally, notice that \begin{equation} s_ias_i^{-1}\in A \ \mathrm{for \ all} \  1\leq i\leq d-1 \end{equation} because otherwise we can replace $S$ with $S_1 = (S\backslash \{a\})\cup \{s_ias_i^{-1}\}$ and since $s_ias_i^{-1}\notin B$ (because $as_i^{-1}\in B$)  we would have $S_1\cap (A\cup B) = \emptyset $.
From $(1),(2),(3), (4)$ and $(5)$ we get $$N_G(H_1,a^2)\subseteq A\cap B \footnote{Here, $N_G(H_1,a^2)$ denotes the normal closure of the subset $H_1\cup \{a^2\}$}$$ 

\medskip 

Also, in the quotient $G/N_G(H_1,a^2)$ we have $\overline{s_i} = \overline{s_j}, 1\leq i, j\leq d-1$ (i.e. the images of $s_i$ and $s_j$ are equal) thus by taking $b = s_1$ we obtain that this quotient is isomorphic to a quotient of infinite dihedral group,  $$D_{\infty} =\langle \bar{a},\bar{b}|\bar{a}^2=\bar{b}^2=e\rangle $$ where $N_G(H_1,a^2)$ is the normal closure of $\{H_1,a^2\}$ in $G$. So, we get a  quotient of $G$ which is isomorphic to a quotient of $D_{\infty}$ but any such quotient is isomorphic to $D_n, n\in (\mathbb{N}\backslash \{0, 1\})\cup \{\infty \}$ which contradicts our assumption.
\end{proof}

{\em From the proof},  we immediately obtain the following proposition which will be of important use in the next section.

\begin{prop} \label{thm:atmostone} Let $G$ be a finitely generated group, $A, B$ be proper subgroups of $G$. Then there exists a finite generating set $S$ of $G$ with $|S\cap (A\cup B)| \leq 1 $ and $\min \{|S\cap (A\backslash B)|, |S\cap (B\backslash A)|\} = 0$.
\end{prop}

 Indeed, in the proof of Claim 3, we notice that $s_i\notin B, 1\leq i\leq d-1$, because (since $as_i\in B$) otherwise $a\in B$ which contradicts our assumption there. $\square $

  \medskip 

\begin{rem} \label{thm:KL} For the group $D_{\infty } = \langle a, b \ | \ a^2 = b^2 = 1\rangle $ with subgroups $K = \langle a, bab \rangle ,  L = \langle b, aba \rangle $ there exists no finite generating set $S$ of $G$ such that $S\cap (K\cup L) = \emptyset $. Indeed, $K\cup L$ contains all words of odd length of $D_{\infty }$ hence no generating set of $D_{\infty }$ is contained in $D_{\infty }\backslash (K\cup L)$. Also, the subgroups $K$ and $L$ are both maximal subgroups of $D_{\infty }$.
\end{rem}

\medskip

 It is useful to introduce the following property of finitely generated groups: 
 
 \medskip 
 
 $(P)$ We say a finitely generated group has property $(P)$  if for all proper subgroups $A,B \leq G$, the group $G$ admits a finite generating set $S$ such that $S\subseteq G\backslash (A\cup B)$.

  \medskip 
  
  Let us note that the infinite dihedral group $D_{\infty }$ has an obvious quotient isomorphic to  Klein's Vierergruppe $\Z /2\Z \times \Z/2\Z$. On the other hand, if $G$ surjects onto  $\Z /2\Z \times \Z/2\Z$, then by taking $A, B$ as preimages of  $(1,0)$ and $(0,1)$ respectively, we see hat there is no generating set $S$ of $G$ which is in $G\backslash (A\cup B)$ since the latter set is in the preimage of $(1,1)$. Thus, combined with this observation, Proposition \ref{thm:AB} can be stated  in the following  cleaner form: 
 
 \begin{prop} \label{thm:ABnew}  A finitely generated group has property $(P)$ if and only if it does not surject onto the Klein's Vierergruppe.
 \end{prop}   
  
 \begin{rem} Proposition \ref{thm:ABnew} (hence also Proposition \ref{thm:AB}, but not Proposition \ref {thm:atmostone}!) also follows from the classical result of Scorza (\cite{Bh}, \cite {Z} )  which states that a group is a union of three proper subgroups if and only if it has a quotient isomorphic to Klein's Vierergruppe (we are grateful to the anonymous referee for pointing this out). Indeed, we just need to show the ``if`` direction,  i.e. let $G$ be a finitely generated group having no quotient isomorphic to $\Z /2\Z \times \Z/2\Z$ and let $A, B$ be proper subgroups. Let also $S_0$ be a finite generating set of $G$.  By Scorza's Theorem, the subgroup of $G$ generated by $ G\backslash (A\cup B)$ is not proper (because this subgroup, together with $A$ and $B$ cover the entire $G$) hence the set  $ G\backslash (A\cup B)$ generates $G$. Thus $G\backslash (A\cup B)$ has a finite subset $S$ which generates every element of $S_0$. Then $S$ generates $G$ and $S\subset G\backslash (A\cup B)$. 
 \end{rem}
 
 \medskip

\section{Proof of Theorem \ref{thm:proper}}

 First, we prove the following claim which shows the use of  Proposition \ref{thm:AB}. 
 
 \begin{prop} \label{thm:generic} Let $G$ be a group with a finite generating set $S$, and $G^{\ast}_{\phi}=(G,A,B,t)$ be an HNN extension where $A, B$ are proper subgroups and $(A\cup B)\cap S = \emptyset $. Then $girth(G_{\phi}^{\ast}) = \infty $.     
 \end{prop}
 
  \begin{proof} The proof is a direct application of Britton's Lemma \cite{Brit}. \footnote{Britton's Lemma states that in an HNN extension $(G,A,B,t)$, if a word $w$ can be expressed $w = g_0t^{\epsilon _1}g_1\dots ...t^{\epsilon _1}g_n, n \geq 1$, with no subwords of the form $t^{-1}g_it, g_i\in A$, or $tg_jt^{-1}, g_j\in B$, then $w\neq 1$.}    Letting $S = \{s_1, \dots , s_n\}$ where $1\notin S$ for any $r\geq 2$, we can take $S^{(r)} = \{t, t^rs_1t^{-2r}, t^{3r}s_2t^{-4r}, \dots , t^{(2n-1)r}s_nt^{-2nr}\}$. By Britton's Lemma, $girth(G^{\ast}_{\phi}, S^{(r)}) \geq r$ and since $r\geq 2$ is arbitrary, we conclude that $girth(G^{\ast}_{\phi}) = \infty $.
  \end{proof}
  
 We already know that we cannot always satisfy the condition of Proposition \ref{thm:generic}. The following proposition takes care of the situation not covered by  Proposition \ref{thm:generic}.  

 \begin{prop} \label{thm:nongeneric} Let $G$ be a group with a finite generating set $S$, and $G^{\ast}_{\phi}=(G,A,B,t)$ be an HNN extension where $A, B$ are proper subgroups and $|(A\cup B)\cap S| = 1$ and $\min \{|S\cap (A\backslash B)|, |S\cap (B\backslash A)|\} = 0$. Then $girth(G_{\phi}^{\ast}) = \infty $.   

\end{prop}

\begin{proof}  Let $S = \{s_1, \dots , s_n\}$ where $s_i\notin A\cup B, 1\leq i\leq n-1$. Without loss of generality, we may assume that $|S\cap (B\backslash A)|\} = 0$. Then $s_n\notin B$. By Britton's Lemma, for any $r\geq 2$, there exists no relation of length less than $r$ among the elements of the generating set 
$$S^{(r)} = \{t, t^{-r}s_1t^{2r}, t^{-3r}s_2t^{4r}, \dots , t^{-(2n-3)r}s_{n-1}t^{(2n-2)r},  u^rs_nu^{-2r}\}$$ where $u = t^{-(2n-3)r}s_{n-1}t^{(2n-2)r}$.

  Indeed, in any reduced non-trivial word of length less than $r$ in the alphabet of $S^{(r)}$ , written as a reduced word in the alphabet $\{t, s_1, \dots , s_n\}$, there is no subword of the form $t^{-1}s_n^{\pm 1}t$ hence Britton's lemma still applies.    
\end{proof}

\medskip 

 The proof of Theorem \ref{thm:proper} follows from Proposition \ref{thm:AB}, Proposition \ref{thm:atmostone}, Proposition \ref{thm:generic} and Proposition \ref{thm:nongeneric}. 

\medskip

   Although we are done with the proof of Theorem \ref{thm:proper}, we still would like to prove separately in more explicit terms the fact that a proper HNN extension of $D_{\infty}$ has infinite girth. Notice that $D_{\infty }$ is an example of a group when we are unable to separate the generating set $S$ from the union $A\cup B$.

\begin{prop} \label{thm:D}
Let $D_q, q\in (\mathbb{N}\backslash \{0, 1\})\cup \{\infty \}$ be a dihedral group with standard generators $a, b$, $A$ and $B$ be two proper isomorphic subgroups. Then, $girth((D_q, A,B,t))=\infty$ for all  $q\in (\mathbb{N}\backslash \{0, 1\})\cup \{\infty \}$ .
\end{prop}
\begin{proof}
We break the proof into the following two cases:
\medskip

\textbf{Case 1:} If $A$ and $B$ are proper cyclic isomorphic subgroups of $D_q$, then the hypothesis of Proposition \ref{thm:generic} can still be arranged. It is easy to see that any proper cyclic subgroup of $D_q$ is isomorphic either to a cyclic group of size $q$ or to $\mathbb{Z}/2\mathbb{Z}$. Let $A=\langle w_1\rangle; B=\langle w_2 \rangle$ be two proper cyclic subgroups of $D_q$. Then we have the following sub-cases:
\medskip

\textbf{Sub-case 1:} For the word length $l(w_1)>1, l(w_2)>1$, $S=\{a,b\}$ works, that is $S\cap (A\cup B)=\emptyset$.
\medskip

\textbf{Sub-case 2:} For $l(w_1)=1, l(w_2)>1$, let $w_1=\{a\}$ then the generating set $S=\{b,ab\}$ does not intersects $A$ and $B$. 
\medskip

\textbf{Sub-case 3:} For $l(w_1)=1, l(w_2)=1$. Choose $S=\{aba,ab\}$. 
\medskip

Hence, we see that the arrangements of Proposition \ref{thm:generic} still happen in the above sub-cases. 
\medskip

\medskip

\textbf{Case 2:} $A$ and $B$ are non-cyclic isomorphic subgroups of $D_q$. 

\medskip 

 Any non-cyclic subgroup of $D_{\infty}$ is isomorphic to $D_{\infty}$; such a subgroup will be of the form $G_{m,n} = \langle a(ba)^m, b(ab)^n \rangle $ where $m,n\in \mathbb{N}\cup \{0\}$. Notice that $G_{0,0} = D_\infty$ and all other subgroups are proper. Similarly, any non-cyclic subgroup of $D_q$ for a finite $q\geq 2$ is isomorphic to $D_{q_1}$ where $q_1|q$, and these subgroups are also of the form $G_{m,n} = \langle a(ba)^m, b(ab)^n \rangle $ for some $m,n\geq 0$. Let $S = \{a, b\}, A= G_{m,n}, B = G_{k,l}$ where $\max \{m,n\} \geq 1$ and $\max \{k,l\} \geq 1$. 
 
 \medskip 
 
  If $A\cap S = B\cap S = \emptyset $, then the hypothesis of Proposition \ref{thm:generic} is satisfied hence  $girth((D_{q}, A,B,t))=\infty$.
 
 \medskip 
 
  If $A\cap S = \emptyset $ but $B\cap S \neq \emptyset $, then without loss of generality we may assume $b\in B\cap S$. Then $a\notin S$ and replacing $S$ with $S' = \{a,ab\}$ we again satisfy the hypothesis of Proposition \ref{thm:AB}. 
  
  \medskip 
  
   The case of $B\cap S = \emptyset $ but $A\cap S \neq \emptyset $ is treated similarly. Thus we are left the case when $A\cap S \neq \emptyset $ and $B\cap S \neq \emptyset $. Then we may assume $a\in A, b\in B$ and $t^{-1}at = u, tbt^{-1} = v$ where $u = a(ba)^m, v = b(ab)^n, m, n\geq 1$. For every $r\geq 1$, we let $S^{(r)} = \{t, t^rat^{-2r}, t^{-r}bt^{2r}\}$. We consider words in the alphabet $t, X = t^rat^{-2r}, Y = t^{-r}bt^{2r}$. 
   
   \medskip 
   
   Notice that for all $n\geq 1$, $$X^n = t^rW_1t^{-(n+1)r}, X^{-n} = t^{(n+1)r}W_2t^{-r}, Y^{-n} = t^rW_3t^{-(n+1)r}, Y^{n} = t^{(n+1)r}W_4t^{-r}$$ where $W_i, 1\leq i\leq 4$ are {\em suitable} in the sense that it can be written as $W_i = u_1t^{p_1}u_2t^{p_2}\dots u_kt^{p_k}u_{k+1}$ with $p_i, 1\leq i\leq k$ being non-zero integers such that if $u_i\in A$, then $p_{i-1} > 0$ (if $i\geq 1), p_i < 0$, and if $u_i\in B$, then $p_{i-1} < 0$ (if $i\geq 1), p_i >0$. Then any word of length less than $r$ in the alphabet $\{t,X,Y\}$ will be still suitable, hence by Britton's Lemma such a word is not identity.     
\end{proof}

\section{Proof of Proposition \ref{thm:semiproper}}

 Given a semi-proper HNN extension $\Gamma = (G, A, B, t)$ with $A=G$, we can form a union $\mathcal{G} = \displaystyle \mathop{\bigcup }_{n\in \mathbb{Z}}t^nGt^{-n}$. Notice that, since $tGt^{-1} = B$, we will have a two-sided infinite chain $$\dots < t^{-2}Gt^2 < t^{-1}Gt < G < tGt^{-1} < t^2Gt^{-2} < \dots $$ of strict inclusions. Then  $\mathcal{G}$ is a normal subgroup and $\Gamma /\mathcal{G} \cong \mathbb{Z}$. 
 
 \medskip 
 
 Now, let $G$ be a group satisfying a law. Then $G$ satisfies a law $W(x,y)$ in two variables. To show the claim of Proposition \ref{thm:semiproper}, we just need to observe that then for all $n\in \mathbb{Z}$, $t^nGt^{-n}$ also satisfies the law $W(x,y)$ hence by the strict inclusions, the normal subgroup $\mathcal{G}$ satisfies $W(x,y)$ as well. On the other hand, the infinite cyclic group $\mathbb{Z}$ also satisfies a law (e.g. $[x,y] = 1$ in $\mathbb{Z}$). It remains to notice that for any short exact sequence $1\to L \to H \to K \to 1$ of groups, $H$ satisfies a law iff $L$ and $K$ satisfy a law. This finishes the proof of Proposition \ref{thm:semiproper}. 
 
 \medskip 
 
 In the above proof we indeed realized the HNN extension $(G, A, B, t)$ as a semi-direct product $\mathbb{Z}\ltimes_{t} \mathcal{G}$. The normal subgroup $\mathcal{G}$ is a particularly meaningful object in the case when $G$ is a finitely-generated nilpotent group. For example, if $G$ is also torsion-free, then $\mathcal{G}$ naturally lies inside the Malcev completion of $G$ (see \cite{Mal}) and one can also treat a general case of a finitely-generated nilpotent group with possibly some torsions.   
 
  \medskip 
  
 Let us recall that torsion elements of a nilpotent group $G$ form a subgroup, called {\em the torsion subgroup}. We will write $Tor(G)$ to denote the torsion subgroup. In addition, if $G$ is finitely generated then the torsion subgroup $Tor(G)$ is finite and normal, moreover, the quotient $G/Tor(G)$ is torsion-free. Now, if we have isomorphic subgroups $A, B \leq G$ with $A = G$, then necessarily $Tor(B) = Tor(A) = Tor(G)$. Hence, for a semi-proper HNN extension $(G, A, B, t)$ where the conjugation by $t$ is given by an isomorphism $\phi : A\to B$ (i.e. by $\phi : G\to B$) then $\phi |_{Tor(G)}:Tor(G)\to Tor(B)$ is an isomorphism and we also obtain an induced isomorphism $\phi _1:G/Tor(G)\to B/Tor(G) $. Then the group $(G, A, B, t)$ admits a normal subgroup $\mathbb{Z}\ltimes_{\phi } Tor(G)$ whereas the quotient by this normal subgroup is isomorphic to $(G/Tor(G), A/Tor(G), B/Tor(G), t_1)$ where the conjugation by $t_1$ is given by the isomorphism $\phi _1$.
 
 \medskip 
 
  Let $G_1 = G/Tor(G), A_1 = A/Tor(G), B_1 = B/Tor(G)$. We consider the HNN extension $(G_1, A_1, B_1, t_1)$ given by the isomorphism $\phi _1$. Notice that this is a semi-proper HNN extension since $A_1 = G_1$, moreover, $G_1$ is torsion-free. Now we use the fact that a finitely generated torsion-free nilpotent group $H$ admits a Malcev completion $\overline{H}$ which is also nilpotent of the same nilpotency degree, moreover, any monomorphism $\psi :H\to H$ can be extended to an isomorphism $\overline{\psi }:\overline{H}\to \overline{H}$. Thus the monomorphism $\phi _1:G_1\to G_1$ can be extended to an isomorphism $\overline{\phi _1}:\overline{G_1}\to \overline{G_1}$. Then the HNN extension $(G_1, A_1, B_1, t_1)$ is a subgroup of $(\overline{G_1}, \overline{A_1}, \overline{B_1}, t_2)$ where the conjugation by $t_2$ is given by the isomorphism $\overline{\phi _1}$. Hence the group $(\overline{G_1}, \overline{A_1}, \overline{B_1}, t_2)$ is a semidirect product $\mathbb{Z}\ltimes_{\overline{\phi _1}} \overline{G_1}$. Hence the original HNN extension $(G, A, B, t)$ is a subgroup of a nilpotent extension of a nilpotent group. Since $(G, A, B, t)$ is not infinite cyclic, in particular, we again see that it has a finite girth.
  
  \medskip

 A great example of a Malcev completion can be described for an integral Heisenberg group $H_{\mathbb{Z}}=<x,y|[[x,y],x]=[[x,y],y]=1>$. This group is isomorphic to the group $U_3(\mathbb{Z})$ of integral unipotent matrices of size $3\times 3$. The Malcev closure of $H_{\mathbb{Z}}$ will be equal to $U_3(\mathbb{R})$, the group of real unipotent matrices of size $3\times 3$.  A semi-proper HNN extension of $H_{\mathbb{Z}}$ will be a subgroup of $\mathbb{Z}\ltimes_{\phi} U_3(\mathbb{R})$. 

\medskip 

However, note that for proper extensions of nilpotent groups this construction fails. Indeed, the following proper HNN extension of $\mathbb{Z}^2=\langle a,b\rangle$ has infinite girth, which is in support to our Theorem \ref{thm:proper}, $(\mathbb{Z}^2,\langle b^{-1} \rangle, \langle a^nb^{-1} \rangle,t)$ for any $n\in \mathbb{Z}$ with $\phi$, $$\phi:\langle b^{-1} \rangle \to \langle a^nb^{-1}\rangle$$ Then, $$(\mathbb{Z}^2)^{\ast}_{\phi}=\langle a,b,t|[a,b]=e,t^{-1}b^{-1}t=a^nb^{-1} \rangle=\langle a,b,t| b^{-1}ab=a, b^{-1}tb=ta^n \rangle=\mathbb{F}_2\rtimes_{\phi} \mathbb{Z}$$ But, it follows from Proposition \ref{thm:free} that $$girth(\mathbb{F}_2\rtimes_{\phi} \mathbb{Z})=\infty$$ 

\section{Proof of Proposition \ref{thm:amalgamation}}

 Given some groups $A, B, C$ with monomorphisms $\phi :C\to A$ and $\psi :C\to B$, one can form a product of $A$ and $B$ amalgamated over $C$. We will write this as $A*_CB$ dropping $\phi $ and $\psi $ from the notation as they will be given to us in the context. It turns out that $A, B, C$ will have isomorphic images in $G = A*_CB$ which we still denote with the same letters. We will be using the following well known analog of Britton's Lemma for amalgamated free products: Let $T_A = A\backslash C, T_B = B\backslash C$ and $w = g_0g_1\dots ...g_n, n \geq 1$ such that for all $1\leq  i \leq n$ if 
 
 (i) if $g_{i-1}\in A$, then $g_i\in T_B$
 
 (ii) if  $g_{i-1}\in B$, then $g_i\in T_A$
 
 (iii) $g_0\neq 1$

 Then $w\neq 1\in  A*_CB$.

 \medskip 
 
 We will call the amalgamated free product {\em proper} if $C$ is a proper subgroup of both $A$ and $B$. Notice that if $C = A$ ($C = B$) then $G$ becomes isomorphic to $B$ (to $A$) so the girth$(G) = \infty $ iff girth$(B) = \infty $ (girth$(A) = \infty $).

 \medskip
 
Let $G = A*_CB$ where $A, B$ are finitely generated groups. If $C$ is trivial, then $G = A*B$ and in this simpler case we can procced as follows: Let $S_1 = \{a_1, \dots , a_n\},  S_2 = \{b_1, \dots , b_m\}$ be generating sets of $A$ and $B$ respectively where $1\notin S_1, 1\notin S_2, S_i\cap S_i^{-1} = \emptyset , 1\leq i\leq 2$. If $A$ and $B$ are both cyclic groups then the claim is an easy exercise (alternatively, the group $A*B$ is word hyperbolic hence the result about the girth follows from Theorem 2.6 in \cite{azer2}), so we will assume that at least one of them, say $A$, is not cyclic. Then $n\geq 2$.  Let now $r\geq 1$. We take reduced words $U_1(X, Y, Z), \dots , U_{m+n}(X, Y, Z), V_1(X,Y,Z), \dots V_{m+n}(X,Y,Z)$ in the free group formally generated by letters $X, Y, Z$ such that for all $1\leq i, j\leq m+n$
  
  (i) $U_i$ ends with $Y$ and $V_j$ starts with $Y^{-1}$;
  
   (ii) $|U_i|=2p, |V_j| = 2p, |U_i^{\epsilon }V_j^{\delta }| > \frac{3}{2}p$ for all $\epsilon , \delta \in \{-1,1\}$ where $p > 2r$; 

(iii) $U_i$ and $V_j$ are reduced words of length at least $p$ in the alphabet $\{ \xi , \eta \}$ where $\xi = XY, \eta = ZY$.
   
   \medskip

   Now we let $S^{(r)} = S^{(r)}_1\cup S^{(r)}_2\cup \{uv, wv\}$ where $$S^{(r)}_1 = \{U_1(u,v,w)a_1V_1(u,v,w), \dots , U_n(u,v,w)a_nV_n(u,v,w)\}$$ and  $$S^{(r)}_2 = \{U_{n+1}(u,v,w)a_1b_{1}a_1V_{n+1}(u,v,w), \dots , U_{m}(u,v,w)a_1b_ma_1V_{m}(u,v,w)\}$$ where $u = a_1b_1a_1^{-1}, v = a_2b_1a_2^{-1}, w = a_1a_2b_1a_2^{-1}a_1^{-1}$.   
   
   \medskip 

 By condition (iii) we have that $U_i, V_i, 1\leq i, j \leq m+n$ are also words of length at least $p$ in $uv$ and $wv$. Then  $S^{(r)}$ is a generator of $G = A*_CB$. Notice that none of the words $$a_1, a_2, a_1^{-1}a_2, a_2a_1^{-1}, a_1^{-1}a_2a_1, a_1a_2a_1^{-1}, a_2^{-1}a_1a_2, a_2a_1a_2^{-1}$$ represents identity element in $A$, hence there is no relation of length less than $r$ among the elements of $S^{(r)}$. Thus girth$(A*B,  S^{(r)}) \geq r$.

\medskip 
   
  In the general case, when $C\neq 1$, by the index assumption of our proposition, without loss of generality, we may assume that $A:C \geq 3$. 

  We will use the following simple lemma.
  
  \begin{lem} \label{thm:involution} Let $G$ be a group and $H\leq G$ be a subgroup such that for all $x\in G\backslash H, x^2\in H$. Then $H$ is a normal subgroup. 
  \end{lem} 

 \begin{proof} Indeed, let $h\in H, x\in G\backslash H$. $xh\in G\backslash H$ hence $(xh)^2\in H$ which yields $xhx\in H$. Then $xhx^{-1} = (xhx)x^{-2}\in H$.
 \end{proof}
 
  Using Lemma \ref{thm:involution}, we can claim another simple lemma. 
  
  \begin{lem} \label{thm:generset} Let $A$ be a group and $C\leq A$ with $A:C\geq 3$. Then there exists distinct $a_1, a_2\in A$ such that $a_1, a_2, a_1^{-1}a_2, a_2a_1^{-1}\notin C$ and $a_1^{-1}a_2a_1, a_1a_2a_1^{-1}, a_2^{-1}a_1a_2, a_2a_1a_2^{-1} \notin C$.
  \end{lem}
  
   \begin{proof} If $C$ is not trivial and there exists $x\in A\backslash C$ such that  $x^2\notin C$, then we can take $a_1 = x, a_2 = x^2$; otherwise, by Lemma \ref{thm:involution}, $C$ is a normal subgroup thus we just need to find two distinct elements $d_1, d_2 \in (G/C)\backslash \{1\}$ satisfying conditions $d_1^{-1}d_2\neq 1$ but such elements trivially exist if $|G/C| \geq 3$.   
   \end{proof}

Then we can choose generating sets $S_1 = \{a_1, \dots , a_n\},  S_2 = \{b_1, \dots , b_m\}$ of $A$ and $B$ such that $S_1\cap C = S_2\cap C = \emptyset $, $n\geq 2$ and none of the elements $$a_1, a_2, a_1^{-1}a_2, a_2a_1^{-1},  a_1^{-1}a_2a_1, a_1a_2a_1^{-1}, a_2^{-1}a_1a_2, a_2a_1a_2^{-1}$$ belong to $C$. Then we define  $S^{(r)}$ as above and there will be no relation of length less than $r$ among the elements of  $S^{(r)}$. Thus we again obtain that  girth$(A*_CB,  S^{(r)}) \geq r$. Since $r$ is arbitrary, we conclude that  girth$(A*_CB) = \infty $.

  \begin{rem} Notice that a proper amalgamated free product $G = A*_CB$ is virtually solvable iff $C$ is a virtually solvable normal subgroup and $G/C\cong D_{\infty }$ so Proposition \ref{thm:amalgamation} indeed confirms the Girth Alternative for the class of proper amalgamated free products.
  \end{rem}

\bigskip 

\section{Proof of Proposition \ref{thm:free}}

Let $\Gamma = (G,A,B,t)$ be an HNN extension of a non-elementary word hyperbolic group $G$.

\medskip

 For proper HNN extension $(G,A,B,t)$, with $\phi:A\to B$ an isomorphism between proper subgroups $A,B<G$, the claim follows from Theorem \ref{thm:proper} that $girth((G,A,B,t))=\infty.$

\medskip

  We will treat the cases of proper and full HNN extensions together. As a major tool, we will consider the actions of word hyperbolic groups on their boundary. Let us recall that the boundary of a word hyperbolic group $G$ is a compact metric space, denoted as $\partial G$. $G$ acts on $\partial G$ by homeomorphisms. Torsion elements of $G\backslash \{1\}$ are called {\em elliptic} and non-torsion elements are called {\em hyperbolic}. A hyperbolic element $g$ has exactly two fixed points on $\partial G$, one of them is attractive and another one is repelling; we will denote these as $P_g$ and $R_g$ respectively.  

  \medskip 

  We will use the following proposition (See \cite{KB}). 

  \begin{prop} \label{thm:hyperbolic} Let $G$ be a non-elementary word hyperbolic group. Then 

  a) The boundary $\partial G$ is infinite;

  b) The sets $\{P_g : g \ \mathrm{is \ a \ hyperbolic \ element \ of} \ G\}$ and $\{R_g : g \ \mathrm{is \ a \ hyperbolic \ element \ of} \ G\}$ are dense in $\partial G$; 

  c) The set $\{(P_g, R_g) : g \ \mathrm{is \ a \ hyperbolic \ element \ of} \ G\}$ is dense in $\partial G\times \partial G$. 

  \end{prop}

  For convenience of the reader we will run the argument for word hyperbolic groups and then extend it to HNN extensions. The following proposition reproves a result (Theorem 2.6)  of \cite{azer2}. Making use of Proposition \ref{thm:hyperbolic}, we offer a simpler proof of this result. 
  
  \begin{prop} A non-elementary word hyperbolic group has infinite girth.  
  \end{prop} 

  \begin{proof}
  Let $\{g_1, \dots , g_s\}$ be a finite generating set of a word hyperbolic group $G$ and $\gamma $ be a hyperbolic element of $G$. By Proposition \ref{thm:hyperbolic}, there exists a hyperbolic element $\beta \in G$ such that $$ (Fix(\beta )\cup \displaystyle \mathop{\cup }_{1\leq i\leq s}g_i^{\pm 1}Fix(\beta )) \cap (Fix(\gamma )\cup \displaystyle \mathop{\cup }_{1\leq i\leq s}g_i^{\pm 1}Fix(\gamma ))  = \emptyset .$$ 

  Since $\partial G$ is a compact metric space, it is Hausdorff, so we can take open neighborhoods $U, V$ of $Fix(\gamma ), Fix(\beta )$ respectively such that $$(U\cup \displaystyle \mathop{\cup }_{1\leq i\leq s}g_i^{\pm 1}U)\cap (V\cup \displaystyle \mathop{\cup }_{1\leq i\leq s}g_i^{\pm 1}V) = \emptyset $$ and the set $$\partial G\backslash (\overline{(U\cup \displaystyle \mathop{\cup }_{1\leq i\leq s}g_i^{\pm 1}U)\cup (V\cup \displaystyle \mathop{\cup }_{1\leq i\leq s}g_i^{\pm 1}V)})$$ is infinite. Let $P$ be any point in this set. 

\medskip 

   Since the fixed points of $\gamma $ and $\beta $ are either attractive or repelling, there exists a natural $N$ such that for all $n\geq N, \gamma ^{\pm n}(\partial G\backslash U) \subseteq U$ and $\beta ^{\pm n}(\partial G\backslash V) \subseteq V$. 

\medskip 

    Then for every $r\geq 1$ we can take a generating set $$S_r = \{\beta ^N, \gamma ^N, \beta ^{Nr}g_1\gamma ^{Nr}, \beta ^{2Nr}g_2\gamma ^{2Nr}, \dots , \beta ^{sNr}g_s\gamma ^{sNr}\}.$$ By our arrangements, for any word $W$ in this alphabet, we will have $W(P)\in U\cup V$, hence $W(P)\neq P$, hence $W\neq 1$. This implies that $girth(G)\geq r$. Since $r$ is arbitrary, we conclude that $girth(G) = \infty $. 
    \end{proof}

\medskip 

    Now we are ready to discuss the HNN extensions of word hyperbolic groups. 

    \medskip 

    Let again $r\geq 1$ and  $\{g_1, \dots , g_s\}$ be a finite generating set of a word hyperbolic group $G$. Again, we can choose hyperbolic elements $\beta , \gamma $ such that the sets $$ A(\beta ) = ( \displaystyle \mathop{\cup }_{0\leq j\leq 2r}Fix(t^{-j}\beta t^{j}))\cup (\displaystyle \mathop{\cup }_{1\leq i\leq s, 0\leq j, l\leq 2r}t^{-j}g_i^{\pm 1}t^{j}Fix(t^{-l}\beta t^{l})) $$ and $$A(\gamma ) = ( \displaystyle \mathop{\cup }_{0\leq j\leq 2r}Fix(t^{-j}\gamma t^{j}))\cup (\displaystyle \mathop{\cup }_{1\leq i\leq s, 0\leq j, l\leq 2r}t^{-j}g_i^{\pm 1}t^{j}Fix(t^{-l}\gamma t^{l})) $$ are disjoint.

    \medskip 

    By letting $g_0=1$ we can conveniently write $$ A(\beta ) =  \displaystyle \mathop{\cup }_{0\leq i\leq s, 0\leq j, l\leq 2r}t^{-j}g_i^{\pm 1}t^{j}Fix(t^{-l}\beta t^{l}) \ \mathrm{and} \ A(\gamma ) =  \displaystyle \mathop{\cup }_{0\leq i\leq s, 0\leq j, l\leq 2r}t^{-j}g_i^{\pm 1}t^{j}Fix(t^{-l}\gamma t^{l}) .$$  The above arrangement will allow us to write $\beta ^i\gamma ^j$ and $\gamma ^j\beta ^i$ as $\beta ^ig_0\gamma ^j$ and $\gamma ^jg_0\beta ^i$ respectively. 
    
    \medskip 

    We can take open neighborhoods $U, V$ of $\displaystyle \mathop{\cup }_{0\leq j\leq 2r}Fix(t^j\beta t^{-i})$ and $\displaystyle \mathop{\cup }_{0\leq j\leq 2r}Fix(t^j\gamma t^{-i})$ respectively such that  $$(U\cup \displaystyle \mathop{\cup }_{1\leq i\leq s}g_i^{\pm 1}U)\cap (V\cup \displaystyle \mathop{\cup }_{1\leq i\leq s}g_i^{\pm 1}V) = \emptyset $$ and the complement $$\partial G\backslash \overline{(U\cup \displaystyle \mathop{\cup }_{1\leq i\leq s}g_i^{\pm 1}U)\cap (V\cup \displaystyle \mathop{\cup }_{1\leq i\leq s}g_i^{\pm 1}V)}$$ is infinite. Again, taking an arbitrary point $P$ in this complement, we observe that there is no relation of length less than $r$ among the elements of $$S = \{t, \beta ^N\gamma ^N, \gamma ^N(\beta ^N\gamma ^N)^r, \beta ^{Nr}g_1\gamma ^{Nr}, \dots , \beta ^{sNr}g_s\gamma ^{sNr}\}.$$

    Indeed, if $W$ is a such a relation, since $\Gamma $ is a semi-proper or full extension \footnote{In a semi-proper or full HNN extension $(G, G, B, t)$, every element can be written as $t^{p}gt^{-q}$ for some $g\in G$ and non-negative integers $p, q$.}, by Britton's Lemma, the sum of exponents of $t$ in $W$ is equal to zero and a cyclic permutation $W'$ of $W$ or $W^{-1}$ can be written as  $$W' = t^q(t^{-i_1}\beta ^{N_1}t^{i_1})(t^{-j_1}g_{l_1}t^{j_1})(t^{-k_1}\beta ^{M_1}t^{k_1})\dots (t^{-i_m}\beta ^{N_m}t^{i_m})(t^{-j_m}g_{l_m}t^{j_m})(t^{-k_m}\beta ^{M_m}t^{k_m})t^{-q}$$
    where $l_1, \dots , l_m\in \{0, 1, \dots , s\}$, all exponents are non-negative, moreover, the exponents $i_1, j_1, k_1, \dots , i_m, j_m, k_m$ are less than $r$, and the exponents $N_1, M_1, \dots , N_m, M_m$ are bigger than $N$. Then by our arrangements, $W'(p)\neq p$, hence $W'\neq 1$, hence $W\neq 1$. 

    \medskip 
    
    Thus $girth (\Gamma ) \geq r$ and since $r$ is arbitrary, we obtain that  $girth (\Gamma ) = \infty $. This completes the proof of Theorem \ref{thm:free}. 
    
  \bigskip
  
   In the case of free groups of rank at least two (i.e. $G\cong \mathbb{F}_n, n\geq 2$), it is not straightforward to find direct elementary proofs. Much studies have been done in recent years about full or semi-proper HNN extensions of non-abelian free groups, yet the Corollary \ref{thm:free1} does not seem to lend itself easily to these results either.

 \begin{rem} Let us emphasize that free-by-cyclic groups $\mathbb{Z}\ltimes_{\phi}\mathbb{F}_n$ are not always hyperbolic so we cannot invoke the result from \cite{azer2} (Theorem 2.6) about hyperbolic groups. In fact, by the result of P.Brinkmann \cite{Brin} these groups are hyperbolic precisely when they are atoroidal, i.e. when they do not contain an isomorphic copy of $\mathbb{Z}\oplus \mathbb{Z}$. Moreover, for $n=2$, $\mathbb{Z}\ltimes_{\phi}\mathbb{F}_n$ is always toroidal, hence not hyperbolic.  Hyperbolicty does not hold for semi-proper HNN extensions either, in fact,    a result of \cite{Mu} states that a semi-proper HNN extension of $\mathbb{F}_n$ is hyperbolic unless it contains a copy of Baumslag-Solitar group $BS(1,m)$ for $m\geq 1$.
 \end{rem}
 
  \medskip 
  
  \begin{rem} One can hope to invoke Theorem 4.4 from \cite{azer2} but linearity is also a problematic issue.  It is not known whether free-by-cyclic groups are linear, although this is known to be true in the case when the rank of the free group is two \cite{Bu3}. For semi-proper HNN extensions,  the presentation $\langle t, a, b  \ | \ t^{-1}at = a^m, t^{-1}bt = b^r  \rangle $ defines a well-known  example studied in \cite{W} and \cite{DS} which, for $m, r\geq 2$, gives a non-linear group. 
  \end{rem}

\medskip 

 It is known that free-by-cyclic groups are either isomorphic to $BS(1,\pm 1)$, or hyperbolic, or large   \cite{Bu4}) and for semi-proper  HNN extensions of a free group, it has been conjectured by J.Button (in personal communications)  that any such group is either hyperbolic or large, unless it is isomorphic to $BS(1,m), m\in \Z \backslash \{0\}$ (a group is large if it has a finite index subgroup which surjects onto a non-abelian free group). This suggests a use of another result from \cite{azer2}, namely,   Proposition 1.1 (combined with Theorem 2.6 there), but again, we have an issue of not being able to deduce infinity of the girth of a group from the infinity of the girth of its finite index subgroup.  Indeed, we would like to ask the following 
 
 \medskip 
 
 {\bf Question 1:} Do finitely generated large groups have infinite girth?
 
 \medskip

 The following general question regarding full HNN extensions of a given group $G$ is also interesting to us:

\medskip 

\textbf{Question 2:} Let $G$ be a non-cyclic finitely generated group. Is it true that $girth(G)=\infty$ if and only if $girth(\mathbb{Z}\ltimes_{\phi} G)=\infty$.
\medskip

If $girth(G) = \infty $, then, in the case of trivial $\Z$-action (i.e. the direct sum), since $G$ is a quotient of $\mathbb{Z}\ltimes_{\phi} G$, by Proposition 1.1.in \cite{azer2}, we obtain that  $girth(\mathbb{Z}\ltimes_{\phi} G)=\infty$. The case of general $\Z$-action remains unclear. We conjecture that in the other direction the answer is negative, i.e. there exists a finitely generated group $G$ and an automorphism $\phi $ such that $girth(G) < \infty $ whereas $girth(\mathbb{Z}\ltimes_{\phi} G)=\infty$. Let us also mention (a somewhat relevant fact) that if $G$ is a finitely generated group satisfying no law in two variables with $girth(G)<\infty$, then $girth(((G\wr\mathbb{Z})\wr\mathbb{Z}))=\infty$ as shown in \cite{azer1}.

 \end{document}